\documentclass[12pt,english]{article}
\usepackage{mathptmx}

\usepackage[T1]{fontenc}
\usepackage[latin9]{inputenc}
\usepackage[letterpaper]{geometry}
\usepackage{float}
\usepackage{amsthm}
\usepackage{amsmath}
\usepackage{graphicx}
\usepackage{amssymb}
\usepackage{esint}

\makeatletter

\providecommand{\tabularnewline}{\\}

  \theoremstyle{plain}
  \newtheorem*{thm*}{Theorem}
\theoremstyle{plain}
\newtheorem{thm}{Theorem}
  \theoremstyle{remark}
  \newtheorem*{rem*}{Remark}
  \theoremstyle{plain}
  \newtheorem{algorithm}[thm]{Algorithm}
  \theoremstyle{plain}
  \newtheorem{lem}[thm]{Lemma}
  \theoremstyle{remark}
  \newtheorem*{acknowledgement*}{Acknowledgement}

\makeatother

\usepackage{babel}

\begin{document}

\title{Reconstruction of Bandlimited Functions from Unsigned Samples}

\author{Gaurav Thakur%
\thanks{Program in Applied and Computational Mathematics, Princeton University,
Princeton, NJ 08544, USA, email: gthakur@princeton.edu%
}}

\date{June 30, 2010}
\maketitle
\begin{abstract}
We consider the recovery of real-valued bandlimited functions from
the absolute values of their samples, possibly spaced nonuniformly.
We show that such a reconstruction is always possible if the function
is sampled at more than twice its Nyquist rate, and may not necessarily
be possible if the samples are taken at less than twice the Nyquist
rate. In the case of uniform samples, we also describe an FFT-based
algorithm to perform the reconstruction. We prove that it converges
exponentially rapidly in the number of samples used and examine its
numerical behavior on some test cases.
\end{abstract}
{\small \par}

\noindent {\small Mathematics Subject Classification (2010): Primary
94A20; Secondary 30D15, 42C15, 94A12}{\small \par}

\section{Introduction\label{SecIntro}}

In a series of recent papers \cite{BBCE09,BCE06,BCE07}, Balan, Casazza
and Edidin have investigated the possibility of reconstructing finite-dimensional
signals using measurements that do not contain sign or phase information.
Motivated by an application in the denoising of speech signals, they
studied \textit{$M$-element frames} for $\mathbb{R}^{n}$, i.e. collections
of $M$ vectors that span $\mathbb{R}^{n}$. They considered $M$-element
frames, $\{f_{k}\}_{1\leq k\leq M}$, $f_{k}\in\mathbb{R}^{n}$, such
that any vector $x\in\mathbb{R}^{n}$ can be uniquely determined from
the inner products $\left\{ \left|\left\langle f_{k},x\right\rangle \right|\right\} _{1\leq k\leq M}$
up to an ambiguity of a sign factor. Using frame theory and combinatorial
methods, they showed that such frames exist if and only if $M\geq2n+1$.
In \cite{BBCE09}, a computational method to carry out this reconstruction
was described in the case where $M\geq\tfrac{n(n+1)}{2}$, using a
special class of frames.\\

It is natural to ask if there are analogous results for continuous-domain
signals, namely in the context of samples of bandlimited functions.
The well-known Whittaker-Shannon-Kotelnikov (WSK) sampling theorem
\cite{Br04} shows that if a bandlimited function $f$ is sampled
at a rate greater than its Nyquist frequency, then it can be uniquely
reconstructed from the samples. However, if the signs of the samples
are unknown, this condition may no longer be sufficient. An example
illustrating this is given by the functions $\sin\left(\pi\left(z+\frac{1}{4}\right)\right)$
and $\cos\left(\pi\left(z+\frac{1}{4}\right)\right)$, which agree
in absolute value at $z=\frac{k}{2}$, $k\in\mathbb{Z}$. On the other
hand, it might be expected that if we oversample $f$ at a sufficiently
high rate, the samples may contain enough redundancy that can we afford
to lose their signs and still recover $f$ up to a global sign factor.
It turns out that this is indeed the case.\\

In this paper, we use a complex variable approach to show that if
a real-valued bandlimited function $f$ is sampled at more than twice
its Nyquist rate, then $f$ can be uniquely determined from the absolute
values of its samples up to a sign factor. Conversely, we find that
if $f$ is sampled at less than twice its Nyquist rate, then it is
not always possible to uniquely determine it in this way. We present
an algorithm to perform this reconstruction, and show that it converges
exponentially rapidly in the number of samples used. We consider a
fairly general class of nonuniformly spaced samples in this paper,
although our numerical approach is developed with uniformly spaced
samples in mind for reasons of computational efficiency.\\

We review some existing theory on nonuniform sampling and bandlimited
functions in Section \ref{SecBG}, and then state and prove our main
theoretical results in Section \ref{SecMain}. We describe our algorithm
and study its convergence properties in Section \ref{SecAlgo}, and
apply it to two test cases in Section \ref{SecNum}.

\section{Background Material\label{SecBG}}

We normalize the Fourier transform as $\hat{f}(\omega)=\int_{-\infty}^{\infty}f(t)e^{-2\pi i\omega t}dt$
for Schwartz functions $f$ and extend it to tempered distributions
in the usual way. For $0<p\leq\infty$, we define the \textit{Paley-Wiener
spaces} of bandlimited functions \cite{Se04} by

\noindent \[
PW_{b}^{p}=\left\{ f\in L^{p}:\mathrm{supp}(\hat{f})\subset[-\tfrac{b}{2},\tfrac{b}{2}]\right\} .\]

\noindent An entire function $g$ is said to be of \textit{exponential
type $b$} if

\noindent \[
b=\inf\left(\beta:|g(z)|\leq e^{\beta|z|},z\in\mathbb{C}\right),\]
and we denote this by writing $\mathrm{type}(g)=b$. By the Paley-Wiener-Schwartz
theorem \cite{Ho03}, $PW_{b}^{p}$ can be equivalently described
as the space of all entire functions $f$ with $\mathrm{type}(f)\leq\pi b$
whose restrictions to $\mathbb{R}$ are in $L^{p}$. It also follows
that $PW_{b}^{p}\subset PW_{b}^{q}$ for $p<q$. Functions $f\in PW_{b}^{p}$
satisfy the classical estimates $\left\Vert f'\right\Vert _{L^{p}}\leq\pi b\left\Vert f\right\Vert _{L^{p}}$
and $\left\Vert f(\cdot+ic)\right\Vert _{L^{p}}\leq e^{\pi b|c|}\left\Vert f\right\Vert _{L^{p}}$,
respectively known as the \textit{Bernstein} and \textit{Plancherel-Polya
inequalities} \cite{Le96,Se04}.\\

We now consider a sequence of points $X=\{x_{k}\}\subset\mathbb{R}$,
indexed so that $x_{k}<x_{k+1}$. For any set $B$, we denote the
number of $x_{k}$ in $B$ by $N(X,B)$. We say that $X$ is \textit{separated}
if $\inf_{k}|x_{k+1}-x_{k}|>0$. Following \cite{DS52}, $X$ is also
said to be \textit{uniformly dense} if it satisfies\begin{equation}
\sup_{k}\left|x_{k}-\frac{k}{d}\right|<\infty\label{Density}\end{equation}

\noindent for some finite $d>0$. We denote this by writing $D(X)=d$,
and by $D(X)=\infty$ if $X$ is not uniformly dense. We will mainly
deal with separated, uniformly dense sequences in this paper. It is
worth mentioning that $D(X)$ is not directly related to Beurling's
upper and lower densities (see \cite{Se04}), and there are sequences
$X$ with finite Beurling densities but for which $D(X)=\infty$.\\

\noindent The \textit{generating function }of a sequence $X\subset\mathbb{R}$
is given by\begin{equation}
S(z)=z^{\delta_{X}}\lim_{r\rightarrow\infty}\prod_{0<|x_{k}|<r}\left(1-\dfrac{z}{x_{k}}\right),\label{GenFunc}\end{equation}
where $\delta_{X}=1$ if $0\in X$ and $\delta_{X}=0$ otherwise.
For a uniform sequence $x_{k}=\tfrac{k}{b},$ $S(z)=\tfrac{\sin(\pi bz)}{\pi b}$.
If $X$ is separated and $D(X)=b<\infty$, then the limit in (\ref{GenFunc})
is finite and the function $S$ lies in the \textit{Cartwright class}
$\, CW_{b}$ \cite{Ko96}, the set of all entire functions $f$ with
$\mathrm{type}(f)\leq\pi b$ that satisfy the growth condition\\

\noindent \[
\int_{-\infty}^{\infty}\dfrac{\max(\log|f(t)|,0)}{t^{2}+1}dt<\infty.\]

\noindent In fact, functions in $CW_{b}$ satisfy the apparently stronger
condition \cite{Ko98}\begin{equation}
\int_{-\infty}^{\infty}\dfrac{\left|\log|f(t)|\right|}{t^{2}+1}dt<\infty.\label{LogInt}\end{equation}

\noindent The following result shows that an arbitrary function in
$CW_{b}$ can be expanded in the form (\ref{GenFunc}) and gives a
useful geometric description of its zeros. \cite{Le96}
\begin{thm*}
\noindent (Cartwright-Levinson) Let $W^{+}(\theta,r)$, $W^{-}(\theta,r)$
and $W'(\theta,r)$ respectively be the wedges $\{z:|z|<r,|\arg z|\leq\theta\}$,
$\{z:|z|<r,|\pi-\arg z|\leq\theta\}$ and $\{z:|z|<r,|\arg z|>\theta,|\pi-\arg z|>\theta\}$.
Suppose \textup{$f\in CW_{b}$, $f\not\equiv0$,} and let $U=\{u_{k}\}$
be the set of its zeros. Define $b'=\frac{1}{\pi}\inf(\mathrm{type}(e^{i\omega z}f(z)),\omega\in\mathbb{R})$.\\

\noindent 1: For any $\theta\in(0,\tfrac{\pi}{2})$, as $r\rightarrow\infty$,\[
\dfrac{N(U,W^{+}(\theta,r))}{r}\rightarrow\dfrac{b'}{2},\,\dfrac{N(U,W^{-}(\theta,r))}{r}\rightarrow\dfrac{b'}{2},\textnormal{\,\ and\,\,}\dfrac{N(U,W'(\theta,r))}{r}\rightarrow0.\]
\\
2: $U$ satisfies \textup{$\sum_{|u_{k}|>0}|\mathrm{Im}(\frac{1}{u_{k}})|<\infty.$}
\\

\noindent 3:\begin{equation}
f(z)=f(0)e^{iqz}z^{\delta_{U}}\lim_{r\rightarrow\infty}\prod_{0<|u_{k}|<r}\left(1-\dfrac{z}{u_{k}}\right),\label{CanonProd}\end{equation}

\noindent for some constant $q$.
\end{thm*}
\noindent If $f$ is real-valued on $\mathbb{R}$, then $b'=b$ and
$q=0$, so in particular, the sequence of real zeros $V\subset U$
of $f$ satisfies $\lim_{r\to\infty}\frac{N(V,[-r,r])}{r}\leq b$.
The expression on the right side of (\ref{CanonProd}) is called a
\textit{canonical product}.\\

\noindent We will also need a second, deeper theorem on $CW_{b}$.
\begin{thm*}
\noindent (Beurling-Malliavin) Let \textup{$f\in CW_{b}$}. Then for
any $\epsilon>0$, there exists $h\in CW_{\epsilon}$, $h\not\equiv0$,
such that $fh\in L^{\infty}$.
\end{thm*}
\noindent A discussion of this result and its significance can be
found in \cite{HMN06}. We can actually choose the above $h$ so that
$fh\in L^{p}$ for any $p>0$, by replacing $h$ with $h\phi$ where
$\phi$ is a function for which $\hat{\phi}$ is smooth and has sufficiently
small support. By considering the function $h(z)\overline{h(\overline{z})}$,
$h$ can also be taken to be real and nonnegative.\\

For $1\leq p\leq\infty$, $X$ is called a \textit{sampling sequence}
for $PW_{b}^{p}$ if there are constants $\alpha_{p,b}$ and $\beta_{p,b}$
such that for all $f\in PW_{b}^{p}$, $\alpha_{p,b}||f(X)||_{l^{p}}\leq\left\Vert f\right\Vert _{L^{p}}\leq\beta_{p,b}\left\Vert f(X)\right\Vert _{l^{p}}$.
In particular, this condition implies that any $f\in PW_{b}^{p}$
is uniquely determined by its samples at $X$. Precise geometric characterizations
of sampling sequences can be very complicated (see \cite{Br04,Se97,Se04}),
but for real, separated and uniformly dense sequences $X$, it is
necessary that $D(X)\geq b$ and sufficient that $D(X)>b$ for $X$
to be a sampling sequence. If $D(X)>b$ and $S$ is the generating
function of $X$, a consequence of the Beurling-Malliavin theorem
is that there is an $h\in CW_{\epsilon}$ with zero set $U$, $\epsilon=D(X)-b$,
such that any $f\in PW_{b}^{\infty}$ can be expressed in terms of
its samples,\\
\begin{equation}
f(z)=\sum_{k=-\infty}^{\infty}f(x_{k})\,\dfrac{S(z)h(x_{k})}{h(z)S'(x_{k})(z-x_{k})},\label{SamplingSeries}\end{equation}

\noindent with uniform convergence on compact subsets of $\mathbb{C}\backslash U$
\cite{Ko96,Se97}. $D(X)$ can be thought of as a generalization of
the {}``sampling rate'' to uniformly dense sequences $X$, and when
$X=\{\frac{k}{s}\}$ is a uniform sequence, the condition $b<D(X)=s$
simply says that $f$ is being oversampled beyond its Nyquist rate.
$h$ can be taken as constant for uniform sequences, and the expansion
(\ref{SamplingSeries}) reduces to the classical WSK sampling theorem.

\section{Main Results\label{SecMain}}

\noindent We fix $p\in(0,\infty]$ for the rest of this section. We
first show that if a bandlimited function is sampled at more than
twice its Nyquist rate, then we can reconstruct it up to a sign factor
from the absolute values of the samples.
\begin{thm}
\label{ThmMain}Let $f\in PW_{b}^{p}$ be real-valued on $\mathbb{R}$,
and let $X\subset\mathbb{R}$ be a separated, uniformly dense sequence
with $D(X)>2b$. Then f can be uniquely determined from $a_{k}=|f(x_{k})|$,
up to a sign factor.\end{thm}
\begin{proof}
\noindent We normalize $b=1$ without loss of generality. Let $S$
be the generating function of $X$ and suppose $h$ is given as in
(\ref{SamplingSeries}). The zeros of $f$ and $h$ are countable,
so we can choose $c>0$ so that $f$ and $h$ have no zeros on the
line $L=\{z:\mathrm{Im}(z)=c\}$. We can write $g=f(\cdot+ic)^{2}\in PW_{2}^{p/2}$
as

\noindent \begin{equation}
g(z)=\sum_{k=-\infty}^{\infty}a_{k}^{2}\,\dfrac{S(z+ic)h(x_{k})}{h(z+ic)S'(x_{k})(z+ic-x_{k})}.\label{MainSamp1}\end{equation}

\noindent From Bernstein's inequality, $g'\in PW_{2}^{p/2}$ and differentiating
(\ref{MainSamp1}) gives a similar expansion for $g'(z)$. Let $a_{k}^{*}=f(x_{k}+ic)$
and choose a point $x_{l}\in X$. Since $f$ has no zeros on $L$,
there is a branch of $\arg f$, which we denote by $\arg_{0}f$, that
is continuous on $L$ and satisfies $\arg_{0}f(x_{l}+ic)\in(-\pi,\pi]$.
We define $\arg_{0}g$ in the same way, and we then have

\noindent \begin{eqnarray}
a_{k}^{*} & = & \exp\left(\frac{1}{2}\log|g(x_{k})|+i\arg_{0}f(x_{k}+ic)\right)\\
 & = & \eta|g(x_{k})|^{1/2}\exp\left(\dfrac{i}{2}\left(\int_{x_{l}}^{x_{k}}\mathrm{Im}\left(\frac{g'(t)}{g(t)}\right)dt+\arg_{0}g(x_{l})\right)\right)\label{fSamp}\end{eqnarray}

\noindent where $\eta=\exp\left(i\left(\frac{1}{2}\arg_{0}g(x_{l})-\arg_{0}f(x_{l}+ic)\right)\right)=\pm1$.
We have now determined samples of $f$, which we can use to express
$f$ as

\noindent \begin{equation}
f(z)=\eta\sum_{k=-\infty}^{\infty}a_{k}^{*}\,\dfrac{S(z-ic)h(x_{k})}{h(z-ic)S'(x_{k})(z-ic-x_{k})}.\label{MainSamp2}\end{equation}
\end{proof}
\begin{rem*}
The proof of Theorem \ref{ThmMain} suggests a three-step procedure
to recover $f$ from $\{a_{k}\}$. We can determine $f^{2}$ from
$\{a_{k}\}$ and take its square root by unwrapping its phase. Since
$f$ will typically have zeros on the real axis, we first move up
in the complex plane using (\ref{MainSamp1}), unwrap the phase there
with (\ref{fSamp}) and then move back to the real axis with (\ref{MainSamp2}).
We will use this approach in Section \ref{SecAlgo}.
\end{rem*}
The next result shows that Theorem \ref{ThmMain} is in a sense sharp.
If we sample a function at less than twice its Nyquist rate, it may
or may not be uniquely determined by the absolute values of the samples,
essentially depending on {}``how many'' of the samples are zero.
\begin{thm}
\label{ThmUnder}Let $X\subset\mathbb{R}$ be a separated sequence
with $D(X)<2b$. Then there is a real-valued function $f\in PW_{b}^{p}$
that cannot be uniquely determined from $|f(x_{k})|$ up to a sign
factor. If in addition $D(X)>\frac{4}{3}b$, then there is another
real-valued function $\tilde{f}\in PW_{b}^{p}$ that can be uniquely
determined from \textup{$|\tilde{f}(x_{k})|$} up to a sign factor,
but for which there is no subsequence $Y\subset X$ with $\tilde{f}(y_{k})=0$
and $\lim_{r\to\infty}\frac{N(Y,[-r,r])}{r}\geq b$.\end{thm}
\begin{proof}
As before, we normalize $b=1$. Suppose $D(X)=2-\epsilon$ for some
$\epsilon>0$, and let $X_{1}=\{x_{2k}\}$ and $X_{2}=\{x_{2k+1}\}$.
It follows from the definition (\ref{Density}) that $D(X_{1})=D(X_{2})=1-\tfrac{\epsilon}{2}$.
Now let $S_{1}$ and $S_{2}$ be the generating functions of $X_{1}$
and $X_{2}$. By the Beurling-Malliavin theorem, we can find $h_{1},h_{2}\in CW_{\epsilon/2}$
such that $S_{1}h_{1}$ and $S_{2}h_{2}$ are in $PW_{1}^{p}$. Then
the functions $f_{1}=S_{1}h_{1}+S_{2}h_{2}$ and $f_{2}=$$S_{1}h_{1}-S_{2}h_{2}$
are also in $PW_{1}^{p}$ and satisfy $|f_{1}(X)|=|f_{2}(X)|$.\\

For the other part of Theorem \ref{ThmUnder}, suppose $\epsilon$
above satisfies $\epsilon<\tfrac{2}{3}$. By countability, we can
choose $c$ so that $h_{3}(z)=h_{1}(z+c)$ has no zeros on $X$. Let
$\tilde{f}=S_{1}h_{3}$, so that among the samples of $\tilde{f}$
at $X$, only the ones at $X_{1}$ are zero. If there is any real-valued
$g\in PW_{1}^{p}$ with $|\tilde{f}(X)|=|g(X)|$, then $\tilde{f}{}^{2}-g^{2}$
is in $PW_{2}^{p/2}$ and has zeros at $X$, so by considering canonical
product expansions, $\tilde{f}{}^{2}-g^{2}=S_{1}S_{2}h_{4}$ for some
$h_{4}\in CW_{\epsilon}$. Since $g^{2}=S_{1}(S_{1}h_{3}^{2}-S_{2}h_{4})$
has an analytic square root, it can only have double zeros, so in
particular, $S_{1}h_{3}^{2}-S_{2}h_{4}$ has to be zero on $X_{1}$.
This implies that $h_{4}$ must have zeros on $X_{1}$, which contradicts
the Cartwright-Levinson theorem because $D(X_{1})=1-\tfrac{\epsilon}{2}>\epsilon$.
So $h_{4}\equiv0$ and $\tilde{f}=\pm g$. \end{proof}
\begin{rem*}
The nonexistence of $Y$ in Theorem \ref{ThmUnder} is what makes
the result interesting, as it means that $\tilde{f}$ cannot be determined
from $Y$ by just using a canonical product. In other words, the nonzero
samples at $X_{2}$ play a role in the uniqueness of $\tilde{f}$.
However, there appears to be no simple characterization of all such
functions $\tilde{f}$ or a numerically useful method of computing
$\tilde{f}$ from $|\tilde{f}(X)|$.
\end{rem*}
$\,$
\begin{rem*}
We have not considered the border case of $D(X)=2b$ in the above
results, in which case the conditions required on the sequence $X$
would become more subtle and depend on the value of $p$. However,
in the elementary case where $x_{k}=\frac{k}{2b}$ is a uniform sequence
and $p=2$, the conclusion of Theorem \ref{ThmMain} still holds by
just using the WSK sampling theorem in place of (\ref{MainSamp1})
and (\ref{MainSamp2}).
\end{rem*}
$\,$

There are no simple analogs of these results if we allow $f\in PW_{b}^{p}$
to be complex-valued. In Theorem \ref{ThmMain}, we used the fact
that when $f$ is real-valued, $f^{2}$ has the same samples as $|f|^{2}$,
but this is no longer the case for complex-valued $f$. In general,
such an $f$ will have complex zeros $u_{k}$ and complex-valued functions
of the form $Bf$, where $B$ is a Blaschke product formed from any
subset of $\{\overline{u_{k}}\}$, will be in $PW_{b}^{p}$ and satisfy
$|Bf|=|f|$ identically on $\mathbb{R}$. If we require all the zeros
of $f$ to be real, then since $f$ can be written as a canonical
product over them, it is simply a modulation of a real-valued function
$g$, i.e. it has the form $f(z)=e^{i(cz+d)}g(z)$ for $c,d\in\mathbb{R}$.
This situation is in contrast to the findings in \cite{BCE06}, where
the types of frames the authors studied had results for complex vectors
comparable to those outlined in Section \ref{SecIntro} for real vectors.

\section{A Reconstruction Algorithm\label{SecAlgo}}

\noindent We now describe how to computationally implement the technique
in the proof of Theorem \ref{ThmMain}. We restrict our attention
to uniform sampling sequences here, as they lead to convolution-type
sampling series that can be calculated efficiently by Fast Fourier
Transform (FFT) methods, but the same ideas can be adapted to the
nonuniform case. We first define the following functions:\begin{eqnarray*}
G(z,M) & = & \frac{\sin(\pi z)}{\pi z}e^{-\frac{\pi}{2M}z^{2}}\\
G'(z,M) & = & \left(\frac{\cos(\pi z)}{\pi z}-\frac{\sin(\pi z)}{M}-\frac{\sin(\pi z)}{\pi z}\right)e^{-\frac{\pi}{2M}z^{2}}\\
G^{*}(z,M) & = & \frac{1}{M}\int_{z-M}^{z}G(t,M)dt\end{eqnarray*}
We also denote the strip $\{z:|\mathrm{Im}(z)|<\delta\}$ by $T_{\delta}$.
Then we have the following results from \cite{SS03}, reproduced here
in slightly different forms for our purposes.
\begin{thm}
\noindent \label{ThmConvFact1}(Schmeisser-Stenger, 2002) Let $f\in PW_{b}^{\infty}$,
$s>b$ and $d<1$. Then for $\mathrm{Re}(z)\in[-dM,dM]$, $f(\frac{z}{s})=\sum_{k=-M}^{M}f(\frac{k}{s})G(z-k,M)+E(z)$,
where $|E(z)|\leq C_{1}M^{-1/2}e^{-\frac{\pi(1-d)}{2}\left(1-\frac{b}{s}\right)M+2\pi|Im(z)|}\cdot\left\Vert f\right\Vert _{L^{\infty}}$
for some constant $C_{1}=C_{1}(\tfrac{b}{s})$.
\end{thm}
\noindent $\,$
\begin{thm}
\noindent \label{ThmConvFact2} (Schmeisser-Stenger, 2002) Suppose
$f$ is analytic in $T_{\delta}$ and $|f(z)|\leq K|z|$ for $z\in T_{\delta}$
and some constant $K$. Then for $z\in\mathbb{R}$,\textup{$-\frac{d}{2}\leq z\leq\frac{d}{2}$,}
$f(z)=\sum_{k=-dM}^{dM}f(\frac{k}{M})G(Mz-k,M)+E(z)$, where $|E(z)|\leq C_{2}K(M/\delta)^{1/2}e^{-\frac{\pi\delta M}{4}}\left\Vert f\right\Vert _{L^{\infty}}$
for some constant $C_{2}$.
\end{thm}
\noindent Theorem \ref{ThmConvFact1} is a form of (\ref{SamplingSeries})
with a non-bandlimited $h$. It converges very rapidly in practice
for about $s\geq1.3b$. Theorem \ref{ThmConvFact2} is a version of
Theorem \ref{ThmConvFact1} as $s\to\infty$.\\

These results lead to the following numerical approach. We impose
a mild restriction to rule out some pathological functions, and assume
for notational convenience that an odd number of samples are used.
\begin{algorithm}
\noindent \label{Algo} Suppose $f\in PW_{b}^{\infty}$ is nonzero
in some strip $\{z:|\mathrm{Im}(z)-c|\leq\delta\}$, $\delta>0$.
Let $s>2b$ and $a_{k}=|f(\frac{k}{s})|$ for $-M\leq k\leq M$.\\

\noindent \textup{1: Compute $g_{M}(\frac{z}{s})=\sum_{k=-M}^{M}a_{k}^{2}\, G(z-k+ic,M)$
at $z=\frac{n}{M}$, $-M^{2}\leq n\leq M^{2}$.}

\noindent \textup{2: Compute $g_{M}^{\prime}(\frac{z}{s})=\sum_{k=-M}^{M}a_{k}^{2}\, G'(z-k+ic,M)$
for $z$ as above.}

\noindent \textup{3: Compute $Q_{M}(n)=\frac{1}{s}\sum_{k=(n-2)M}^{(n+1)M}\mathrm{Im}\left(\dfrac{g_{M}^{\prime}(\frac{k}{Ms})}{g_{M}(\frac{k}{Ms})}\right)G^{*}(Mn-k,M)$
for $-(M-2)\leq n\leq M-1$.}

\noindent \textup{4: Compute $R_{M}(n)$ given by $R_{M}(0)=0$, $R_{M}(n)=R_{M}(n-1)+Q_{M}(n-1)$
for $-(M-1)\leq n\leq M-1$.}

\noindent \textup{5: Compute $f_{M}(\frac{z}{s})=\sum_{k=-(M-1)}^{M-1}\sqrt{|g_{M}(\frac{k}{s})|}e^{\frac{i}{2}(R_{M}(k)+\arg g_{M}(0))}\, G(z-k-ic,M)$.}\\

\end{algorithm}
The convolutions in steps 1 and 2 can be performed by $2M$ FFTs of
size $2M+1$ each. Steps 3 and 4 are most efficiently done by direct
computations that respectively involve $O(M^{2})$ and $O(M)$ operations.
Step 5 involves $2N$ FFTs of size $2M+1$ for some integer $N$,
depending on how finely we want to compute $f$. This gives an overall
complexity of $O(M^{2}\log M)$.\\

Theorem \ref{ThmConvFact2} is used above to calculate the integral
in (\ref{fSamp}). The advantage of this approach is that the error
bound in Theorem \ref{ThmConvFact2} does not depend on the derivatives
of $f$. If we instead used conventional methods of integration such
as Gauss quadrature, we would need additional restrictions on $\delta$
to ensure that the algorithm actually converges. There is no simple
closed-form expression for $G^{*}$ but its numerical calculation
only depends on $M$, so we can tabulate its values by a standard
quadrature method using $O(M^{2})$ operations and reuse them for
different $f$ with a lookup table.\\

\noindent We can establish the following convergence result for Algorithm
\ref{Algo}.
\begin{thm}
\noindent \label{ThmAlgo}Let $f$ and $f_{M}$ be defined as in Algorithm
\ref{Algo}. As $M\rightarrow\infty$, $f_{M}\rightarrow\eta f$ uniformly
on compact subsets of $\mathbb{R}$, where $\eta=\pm1$ and $|f_{M}(z)-\eta f(z)|=O\left(e^{-\min\left(\frac{\pi}{16}\left(1-\frac{2b}{s}\right),\frac{\pi\delta}{8}\right)M+4\pi c}\right)$.
\end{thm}
This estimate is somewhat conservative and larger exponents of convergence
are possible if we make more assumptions on $f$, but it shows how
$\delta$ and $c$ affect the rate of convergence. Many real-world
bandlimited signals have most of their zeros on or near the real axis,
so choosing $c$ too small will result in a small $\delta$ while
choosing $c$ too large will sharply increase the constant in the
error bound. Values of $c$ between about $0.01$ to $0.25$ appear
to work well in practice.\\

\noindent Before we prove Theorem \ref{ThmAlgo}, we will need an
auxiliary lemma.
\begin{lem}
\label{hDh}Let $h\in PW_{b}^{\infty}$ be nonzero in $T_{\delta}$.
Then $\left|\frac{h'(z)}{h(z)}\right|\leq K(|z|^{2}+1)$ in $T_{\delta/2}$
for some constant $K$.\end{lem}
\begin{proof}
\noindent Let $U=\{u_{k}\}$ be the zeros of $h(z-i\delta)$ lying
in the upper half plane $\mathbb{C}^{+}=\{z:\mathrm{Im}(z)>0\}$,
and suppose $\frac{\delta}{2}<\mathrm{Im}(z)<\frac{3\delta}{2}$.
The Plancherel-Polya inequality shows that $h(\cdot-i\delta)\in PW_{b}^{\infty}$
and that the function $e^{\pi ibz}h(z-i\delta)$ is bounded and analytic
on $\mathbb{C}^{+}$, so it has the inner-outer factorization \cite{Ga07}
\[
\log h(z-i\delta)=\frac{1}{\pi i}\int_{-\infty}^{\infty}\frac{1+zt}{(t-z)(t^{2}+1)}\log\left|h(t-i\delta)\right|dt+\log\left(\prod_{k}\frac{|u_{k}^{2}+1|}{u_{k}^{2}+1}\frac{z-u_{k}}{z-\overline{u_{k}}}\right)+Az+B\]

\noindent for some constants $A$ and $B$. We differentiate this
to find that\begin{eqnarray*}
\left|\frac{h'(z-i\delta)}{h(z-i\delta)}\right| & = & \left|\frac{1}{\pi i}\int_{-\infty}^{\infty}\frac{\log\left|h(t-i\delta)\right|}{(t-z)^{2}}dt+\sum_{k}\frac{2i\mathrm{Im}(u_{k})}{(z-u_{k})(z-\overline{u_{k}})}+A\right|\\
 & \leq & (2/\delta)^{2}\left(|\mathrm{Re}(z)|^{2}+1\right)\left(\frac{1}{\pi}\int_{-\infty}^{\infty}\frac{\log\left|h(t-i\delta)\right|}{t^{2}+1}dt+\sum_{k}\frac{2\mathrm{Im}(u_{k})}{|u_{k}|^{2}+1}+|A|\right).\end{eqnarray*}

\noindent Since $PW_{b}^{\infty}\subset CW_{b}$, condition (\ref{LogInt})
shows that the integral in the first term is finite. The Cartwright-Levinson
theorem implies that the sum in the second term also converges, which
finishes the proof.
\end{proof}
\noindent It is possible to obtain sharper results than this, but
this lemma is good enough for our purposes.\\

\begin{proof}[Proof of Theorem \ref{ThmAlgo}]
\noindent We proceed by establishing several intermediate bounds
and then combine them all at the end. To simplify the notation, we
will always use $z\in\mathbb{R}$ and $k\in\mathbb{Z}$ to denote
function arguments within norms, e.g. the norm of $F\in L^{\infty}(\mathbb{R})$
will be denoted by $\left\Vert F(z)\right\Vert $.\\

\noindent Let $g$ and $\eta$ be defined as in the proof of Theorem
\ref{ThmMain}. Define $I_{r,M}=\big[-\left(\lfloor r(M-1)\rfloor-1\right),\lfloor r(M-1)\rfloor-1\big]$,
$i_{r,M}=\mathbb{Z}\bigcap I_{r,M}$ and $j_{r,M}=\mathbb{Z}\bigcap\left[-M\lfloor rM\rfloor,M\lfloor rM\rfloor\right]$.
Theorem \ref{ThmConvFact1} shows that\[
E_{1}:=\left\Vert g_{M}(\tfrac{k}{Ms})-g(\tfrac{k}{Ms})\right\Vert _{l^{\infty}(j_{3/4,M})}\leq\frac{C_{1}}{\sqrt{M}}e^{-\frac{\pi}{8}\left(1-\frac{2b}{s}\right)M+2\pi c}\left\Vert f\right\Vert _{L^{\infty}(\mathbb{R})}^{2}.\]

\noindent The error term $E$ in Theorem \ref{ThmConvFact1} is an
entire function and it satisfies the classical Cauchy estimate\[
\left\Vert E'(z+ic)\right\Vert _{L^{\infty}(I_{3/4,M+2})}\leq2\pi\max\left(\left|E(z+ic)\right|,\left\{ z:|\mathrm{Re}(z)|\leq\frac{3}{4}(M+1),|\mathrm{Im}(z)|\leq\frac{1}{2\pi}\right\} \right).\]

\noindent This shows that\[
E_{2}:=\left\Vert g_{M}^{\prime}(\tfrac{k}{Ms})-g^{\prime}(\tfrac{k}{Ms})\right\Vert _{l^{\infty}(j_{3/4,M})}\leq\frac{2\pi C_{1}}{\sqrt{M}}e^{-\frac{\pi}{8}\left(1-\frac{2b}{s}\right)(M-1)+2\pi c+1}\left\Vert f\right\Vert _{L^{\infty}(\mathbb{R})}^{2}.\]
\\

\noindent A standard identity for Jacobi theta functions (\cite{WW27},
p. 475) gives the bound\[
E_{3}:=\left\Vert \sum_{l=-M}^{M}G(z-l-ic,M)\right\Vert _{L^{\infty}(I_{1/4,M})}\leq\left\Vert \sum_{l=-\infty}^{\infty}e^{\frac{\pi}{2M}(c^{2}-(z-l)^{2})}\right\Vert _{L^{\infty}(\mathbb{R})}\leq\frac{3}{2}e^{\frac{\pi c^{2}}{2M}}\sqrt{2M},\]

\noindent and similarly,\[
E_{4}:=\left\Vert \sum_{l=(k-2)M}^{(k+1)M}|G^{*}(Mk-l,M)|\right\Vert _{l^{\infty}(i_{3/4,M+2})}\leq\frac{3}{2}\sqrt{2M}.\]
\\

\noindent Now by Lemma \ref{hDh}, $\left|\frac{g'(z)}{g(z)}\right|\leq K(|z|^{2}+1)$
in $T_{\delta/2}$ for a constant $K$. Since $\left\Vert g_{M}-g\right\Vert _{L^{\infty}(I_{3/4,M+2})}\to0$
and $g$ has no zeros on $\mathbb{R}$, we have $\left\Vert 1/g_{M}\right\Vert _{L^{\infty}(I_{3/4,M+2})}<\infty$
for sufficiently large $M$. We also have the bound\begin{eqnarray*}
\left\Vert \dfrac{g_{M}'(\frac{k}{Ms})}{g_{M}(\frac{k}{Ms})}-\dfrac{g'(\frac{k}{Ms})}{g(\frac{k}{Ms})}\right\Vert _{l^{\infty}(j_{3/4,M})} & \leq & \left\Vert \frac{1}{g_{M}}\right\Vert _{L^{\infty}(I_{3/4,M+2})}\left(E_{2}+\left\Vert g'\right\Vert _{L^{\infty}(I_{3/4,M+2})}\left\Vert 1-\dfrac{g_{M}}{g}\right\Vert _{L^{\infty}(I_{3/4,M+2})}\right)\\
 & \leq & \left\Vert \frac{1}{g_{M}}\right\Vert _{L^{\infty}(I_{3/4,M+2})}\left(K(M^{2}+1)+1\right)\max(E_{1},E_{2}).\end{eqnarray*}
We can use Theorem \ref{ThmConvFact2} with this to find that\begin{eqnarray*}
E_{5} & := & \left\Vert R_{M}(k)-\int_{0}^{k}\mathrm{Im}\left(\frac{g'(\frac{t}{s})}{g(\frac{t}{s})}\right)dt\right\Vert _{l^{\infty}(i_{3/4,M})}\\
 & \leq & M\left\Vert Q_{M}(k)-\int_{k-1}^{k}\mathrm{Im}\left(\frac{g'(\frac{t}{s})}{g(\frac{t}{s})}\right)dt\right\Vert _{l^{\infty}(i_{3/4,M}\backslash\{1-\lfloor\frac{3}{4}(M-1)\rfloor\})}\\
 & \leq & M\left(\left\Vert \dfrac{g_{M}'(\frac{k}{Ms})}{g_{M}(\frac{k}{Ms})}-\dfrac{g'(\frac{k}{Ms})}{g(\frac{k}{Ms})}\right\Vert _{l^{\infty}(j_{3/4,M})}E_{4}+C_{2}(2M/\delta)^{1/2}e^{-\frac{\pi\delta M}{8}}\left\Vert \dfrac{g'(\frac{k}{Ms})}{g(\frac{k}{Ms})}\right\Vert _{l^{\infty}(j_{3/4,M})}\right)\\
 & \leq & \left\Vert \frac{1}{g_{M}}\right\Vert _{L^{\infty}(I_{3/4,M+2})}\left(K(M^{2}+1)+1\right)M^{3/2}E_{2}+C_{2}KM^{3/2}(2/\delta)^{-1/2}e^{-\frac{\pi\delta M}{8}}.\end{eqnarray*}
We now put everything together. Using the elementary inequality $||u|^{1/2}e^{i\theta}-|v|^{1/2}|\leq|u-v|^{1/2}+|1-e^{i\theta}||v|^{1/2}$
along with (\ref{fSamp}) gives {\allowdisplaybreaks

\noindent \begin{eqnarray*}
E_{6} & := & \left\Vert f_{M}\left(\frac{z}{s}\right)-\eta f\left(\frac{z}{s}\right)\right\Vert _{L^{\infty}(I_{1/4,M})}\\
 & \leq & \left\Vert f_{M}\left(\frac{k}{s}+ic\right)-\eta f\left(\frac{k}{s}+ic\right)\right\Vert _{l^{\infty}(i_{3/4,M})}E_{3}+C_{1}e^{-\frac{\pi}{8}\left(1-\frac{b}{s}\right)(M-1)+2\pi c}\left\Vert f\right\Vert _{L^{\infty}(\mathbb{R})}\\
 &  & +\left\Vert \sum_{l\in i_{1,M}\backslash i_{3/4,M}}e^{-\frac{\pi}{2(M-1)}\left(z-l\right)^{2}+\pi c}\right\Vert _{L^{\infty}(\mathbb{R})}\left\Vert f_{M}\left(\frac{k}{s}+ic\right)\right\Vert _{l^{\infty}(i_{1,M}\backslash i_{3/4,M})}\\
 & \leq & \left\Vert \left|g_{M}\left(\frac{k}{s}\right)\right|^{1/2}e^{\frac{i}{2}(R_{M}(k)+\arg g_{M}(0))}-\eta f\left(\frac{k}{s}+ic\right)\right\Vert _{l^{\infty}(i_{3/4,M})}E_{3}\\
 &  & +2C_{1}e^{-\frac{\pi}{8}\left(1-\frac{b}{s}\right)(M-1)+4\pi c}\left\Vert f\right\Vert _{L^{\infty}(\mathbb{R})}\\
 & \leq & \Bigg(E_{1}^{1/2}+\left\Vert 1-\exp\left(R_{M}(k)-\int_{0}^{k}\mathrm{Im}\left(\frac{g'(t)}{g(t)}\right)dt+\arg g_{M}(0)-\arg g(0)\right)\right\Vert _{l^{\infty}(i_{3/4,M})}\\
 &  & \cdot\left\Vert f\right\Vert _{L^{\infty}(\mathbb{R})}E_{3}\Bigg)+2C_{1}e^{-\frac{\pi}{8}\left(1-\frac{b}{s}\right)(M-1)+4\pi c}\left\Vert f\right\Vert _{L^{\infty}(\mathbb{R})}\\
 & \leq & \left(E_{1}^{1/2}+\frac{1}{2}(E_{5}+E_{1})\left\Vert f\right\Vert _{L^{\infty}(\mathbb{R})}\right)E_{3}+2C_{1}e^{-\frac{\pi}{8}\left(1-\frac{b}{s}\right)(M-1)+4\pi c}\left\Vert f\right\Vert _{L^{\infty}(\mathbb{R})}\\
 & = & C_{3}(f)M^{7/2}\exp\left(-\min\left(\frac{\pi}{16}\left(1-\frac{2b}{s}\right),\frac{\pi\delta}{8}\right)M+\max\left(\frac{\pi c^{2}}{2M},4\pi c\right)\right)\end{eqnarray*}

\noindent } for some $C_{3}(f)<\infty$ and sufficiently large $M$,
which establishes the result.
\end{proof}

\section{Numerical Experiments\label{SecNum}}

We illustrate how Algorithm \ref{Algo} works on two test cases. We
consider the translated Bessel function (see \cite{WW27}) $f(z)=J_{1}(z+20)$
sampled at $z=k$, $-M\leq k\leq M$, and a collection of $2M+1$
samples taken from an 8-bit, 44khz audio file. To show the effect
of the parameter $c$, we take $c=0.1$ in the first example and $c=0.04$
in the second one. Values of $G^{*}$ are tabulated using the built-in
Gauss-Kronrod algorithm in MATLAB. We measure the worst-case error
over $I_{1/2,M+1}$ as defined in Section \ref{SecAlgo}, in order
to avoid influence from inaccuracies close to the boundary of the
full domain $I_{1,M+1}$.

\begin{figure}[H]
\begin{raggedright}
$\,$
\par\end{raggedright}

\begin{raggedright}
$\,$
\par\end{raggedright}

\begin{raggedright}
$\,$
\par\end{raggedright}

\begin{raggedright}
\begin{tabular}{|c|c|}
\hline 
$M$ & Error over $I_{1/2,M+1}$\tabularnewline
\hline 
10 & $3.7490\cdot10^{-2}$\tabularnewline
\hline 
20 & $5.9513\cdot10^{-4}$\tabularnewline
\hline 
30 & $4.0158\cdot10^{-5}$\tabularnewline
\hline 
40 & $3.8732\cdot10^{-6}$\tabularnewline
\hline 
50 & $3.8362\cdot10^{-7}$\tabularnewline
\hline
\end{tabular}\includegraphics[bb=0bp 400bp 0bp 200bp,scale=0.5]{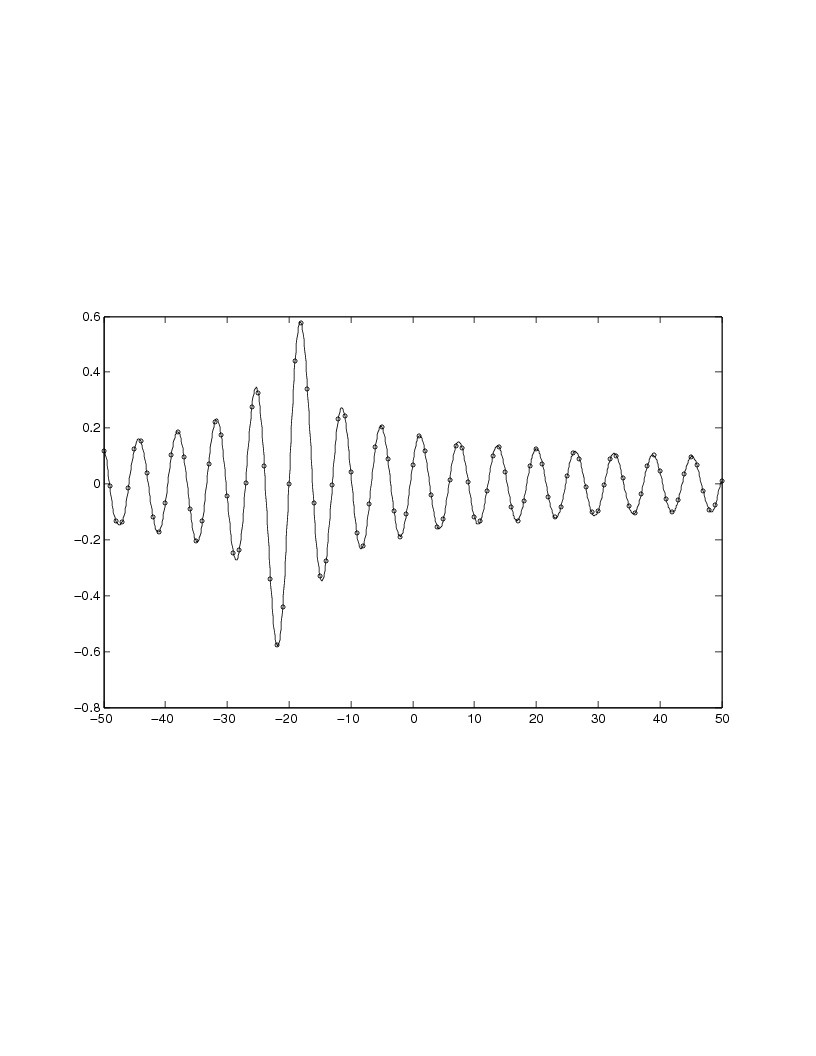}
\par\end{raggedright}

\begin{raggedright}
$\,$
\par\end{raggedright}

\begin{raggedright}
$\,$
\par\end{raggedright}

$\,$

\caption{The graph of $f(z)=J_{1}(z+20)$ and the reconstruction errors with
$c=0.1$. It is clear that the error decays exponentially as $M$
increases.}

\end{figure}

\begin{figure}[H]
$\,$

\begin{raggedright}
$\,$
\par\end{raggedright}

\begin{tabular}{|c|c|}
\hline 
$M$ & Error over $I_{1/2,M+1}$\tabularnewline
\hline 
10 & $1.9752\cdot10^{-2}$\tabularnewline
\hline 
20 & $5.1622\cdot10^{-3}$\tabularnewline
\hline 
30 & $1.5710\cdot10^{-4}$\tabularnewline
\hline 
40 & $1.1563\cdot10^{-4}$\tabularnewline
\hline 
50 & $8.8637\cdot10^{-5}$\tabularnewline
\hline
\end{tabular}\includegraphics[bb=0bp 400bp 0bp 200bp,scale=0.5]{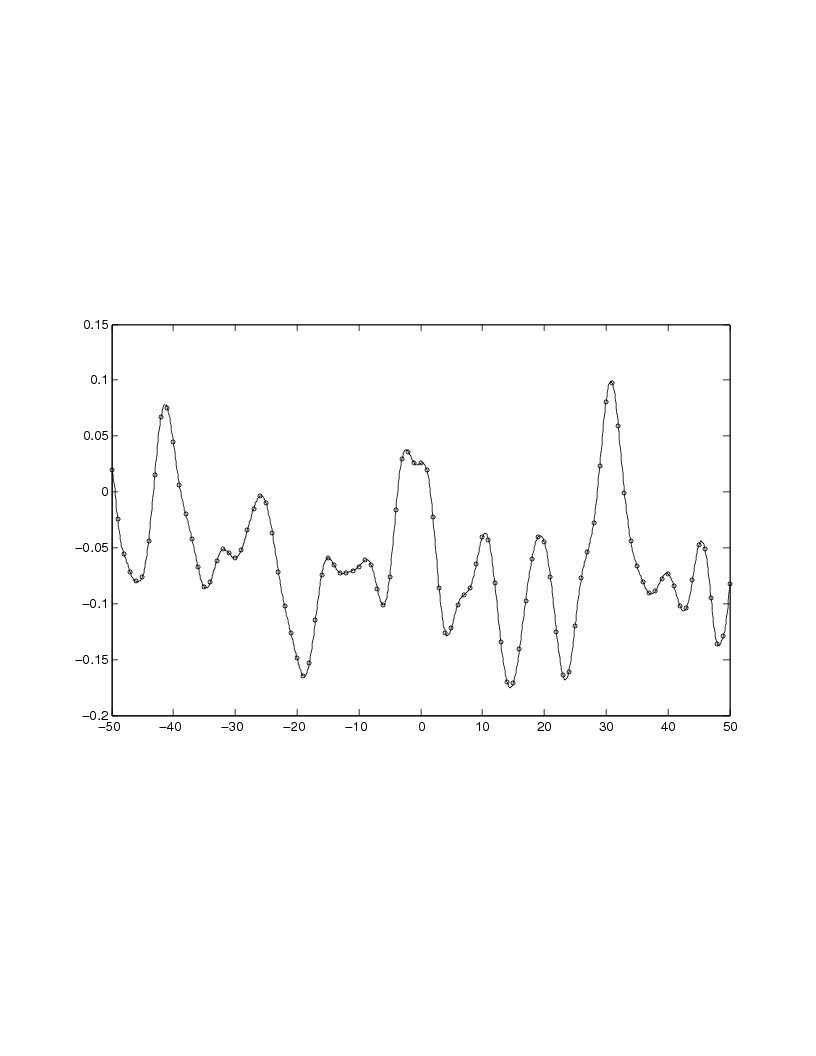}

$\,$

\begin{raggedright}
$\,$
\par\end{raggedright}

$\,$\caption{The graph of a section of audio data and the resulting reconstruction
errors with $c=0.04$. The convergence here is slower than it was
in the preceding example. This is likely due to the presence of complex
zeros with imaginary parts very close to $c$, as well as the proximity
of the real zeros.}

\end{figure}

\begin{acknowledgement*}
The author would like to thank Professor Ingrid Daubechies for many
valuable discussions in the course of this work.
\end{acknowledgement*}
\bibliographystyle{amsplain}
\bibliography{10028}

\end{document}